\documentclass{amsart}
\usepackage[utf8]{inputenc}
\usepackage{amsmath}
\usepackage{amssymb}
\usepackage{graphicx}
\usepackage[marginratio=1:1,height=584pt,width=360pt,tmargin=117pt]{geometry}
\usepackage{amsthm}
\usepackage{amsfonts}
\usepackage{mathtools}
\usepackage{comment}
\usepackage[shortlabels]{enumitem}

\newtheorem{thm}{Theorem}[section]
\newtheorem{lemma}[thm]{Lemma}
\newtheorem{fact}[thm]{Fact}

\newtheorem{cor}[thm]{Corollary}
\newtheorem{prop}[thm]{Proposition}

\newtheorem{defn}[thm]{Definition}

\newtheorem*{question1}{Question 1}
\newtheorem*{question2}{Question 2}
\newtheorem*{question3}{Question 3}
\newtheorem*{question4}{Question 4}

\newcommand{\dom}{\operatorname{dom}}
\usepackage{color}
\newcommand{\red}{\color{red}}

\newcommand{\dcl}{\operatorname{dcl}}

\newcommand{\tp}{\operatorname{tp}}

\newcommand{\calM}{\mathcal{M}}
\newcommand{\calR}{\mathcal{R}}

\newcommand{\calA}{\mathcal{A}}
\newcommand{\calB}{\mathcal{B}}
\newcommand{\calC}{\mathcal{C}}

\newcommand{\calL}{\mathcal{L}}

\newcommand{\calN}{\mathcal{N}}

\newcommand{\N}{\mathbb{N}}
\newcommand{\R}{\mathbb{R}}

\newcommand{\Q}{\mathbb{Q}}
\newcommand{\acl}{\operatorname{acl}}

\newcommand{\interior}{\operatorname{int}}

\newcommand{\Sk}{\operatorname{Sk}}
\newcommand{\eqdf}{=_{\operatorname{df}}}

\newcommand{\mbb}[1]{\mathbb{#1}}
\newcommand{\mc}[1]{\mathcal{#1}}
\renewcommand{\tilde}[1]{\widetilde{#1}}

\newcommand{\into}{\hookrightarrow}

\title{Pathological examples of structures with o-minimal open core}

\author{Alexi Block Gorman}
\address
{Department of Mathematics\\University of Illinois at Urbana-Champaign\\1409 West Green Street\\Urbana, IL 61801}
\email{atb2@illinois.edu}

\author{Erin Caulfield}
\address{Department of Mathematics \& Statistics\\
McMaster University\\
Hamilton Hall\\
1280 Main Street West\\
Hamilton, Ontario, Canada\\
L8S 4K1}
\email{caulfiee@mcmaster.ca}

\author{Philipp Hieronymi}
\address
{Department of Mathematics\\University of Illinois at Urbana-Champaign\\1409 West Green Street\\Urbana, IL 61801}
\email{phierony@illinois.edu}
\urladdr{http://faculty.math.illinois.edu/\textasciitilde phierony}

\date{\today}

\thanks{This is a preprint version. Later versions might contain significant changes.}

\begin{document}

\begin{abstract}
    This paper answers several open questions around structures with o-minimal open core. We construct an expansion of an o-minimal structure $\mathcal{R}$ by a unary predicate such that its open core is a proper o-minimal expansion of $\mathcal{R}$. We give an example of a structure that has an o-minimal open core and the exchange property, yet defines a function whose graph is dense. Finally, we produce an example of a structure that has an o-minimal open core and definable Skolem functions, but is not o-minimal. 
\end{abstract}

\maketitle

\section{Introduction}

Introduced by Miller and Speissegger \cite{MS99} the notion of an open core has become a mainstay of the model-theoretic study of ordered structures. However, there are still many rather basic questions, in particular about structures with o-minimal open cores, that have remained unanswered. In this paper, we are able to settle some of the questions raised by Dolich, Miller and Steinhorn \cite{DMS, DMS2}.\newline

\noindent Throughout this paper, $\mathcal{R}$ denotes a fixed, but arbitrary expansion of a dense linear order $(R,<)$ without endpoints. We now recall several definitions from the aforementioned papers. We denote by $\mathcal{R}^{\circ}$ the structure $(R,(U))$, where $U$ ranges over the open sets of all arities definable in $\mathcal{R}$, and call this structure \textbf{the open core of $\mathcal{R}$}. \newline

\noindent Given two structures $\mathcal{S}_1$ and $\mathcal{S}_2$ with the same universe $S$, we say $\mathcal{S}_1$ and $\mathcal{S}_2$ are \textbf{interdefinable} (short: $\mathcal{S}_1\eqdf \mathcal{S}_2$) if $\mathcal{S}_1$ and $\mathcal{S}_2$ define the same sets. 
For a given theory $T$ extending the theory of dense linear orders, we say that a theory $T'$ is an \textbf{open core of $T$} if for every $\mathcal M \models T$ there exists $\mathcal M' \models T'$ such that $\mathcal M^{\circ}\eqdf \mathcal{M'}.$

\begin{question1}[{\cite[p. 1408]{DMS}}]\label{ques:unary} If $S \subseteq R$ and $(\mathcal{R},S)^{\circ}$ is  o-minimal, is $(\mathcal{R},S)^{\circ} \eqdf \mathcal{R}^{\circ}$?
\end{question1}
\noindent We give a negative answer to this question by constructing an expansion of the real field by a single unary predicate whose open core is o-minimal, but defines an irrational power function. It is clear from the construction in Section 2 that there are similar examples of expansions of the real ordered additive group that do not define an ordered field.\newline

\noindent We say $\mathcal{R}$ is \textbf{definably complete} (short: $\mathcal{R} \models \operatorname{DC}$) if every definable unary set has both a supremum and an infimum in $R\cup \{\pm \infty\}$.
We denote by $\dcl_{\mathcal{R}}$ the definable closure operator in $\mathcal{R}$, and often drop the subscript $\mathcal{R}$. 
We say that $\mathcal R$ has the \textbf{exchange property} (short: $\mathcal{R}\models \operatorname{EP}$) if $b \in \operatorname{dcl}(S\cup \{a\})$ for all $S\subseteq R$ and $a,b\in R$ such that $a\in \operatorname{dcl}(S\cup\{b\})\setminus \operatorname{dcl}(S)$. For a theory $T$, we say $T$ has the exchange property (short: $T\models \operatorname{EP}$) if every model of $T$ has the exchange property.
We write \textbf{$\mathcal{R} \models \operatorname{NIP}$}  if its theory does not have the independence property as introduced by Shelah \cite{Shelah}. We refer the reader to Simon \cite{S15} for a modern treatment of NIP and related model-theoretic tameness notions.

\begin{question2}[{\cite[p. 1409]{DMS}}]\label{DCEPplus} If $\mathcal{R} \models \operatorname{DC} + \operatorname{EP} + \operatorname{NIP}$ and expands an ordered group, is $\mathcal{R}$ o-minimal?
\end{question2}

\noindent By \cite[p. 1374]{DMS} we know that $\mathcal{R}$ has o-minimal open core if $\mathcal{R} \models \operatorname{DC} + \operatorname{EP}$ and expands an ordered group. Thus Question 2 asks whether there is a combinatorical model-theoretic tameness condition that can be added to force o-minimality. Again, we give a negative answer to this question. We construct a  counterexample as follows: Let $\Q(t)$ be the field of rational functions in a single variable $t$. We consider an expansion $\calR_t$ of the ordered real additive group  $(\mathbb{R},<,+,0,1)$ into a $\Q(t)$-vector space such that for all $c\in \Q(t)\setminus \Q$, the graph of the function $x\mapsto cx$ is dense in $\R^2$. We show that $\mathcal{R}_t\models \operatorname{DC} + \operatorname{EP}+\operatorname{NIP}$, but is not o-minimal. \newline
  
\noindent The structure $\calR_t$ has infinite dp-rank. By Simon \cite{Simon-DP}, if $\mathcal{R}\models \operatorname{DC}$, expands an ordered group and has dp-rank 1, then $\mathcal{R}$ is o-minimal. However, we do not know whether $\mathcal{R}$ is o-minimal if $\mathcal{R}\models \operatorname{DC} + \operatorname{EC}$, expands an ordered group and has finite dp-rank.\newline

\noindent In addition to showing that $\calR_t\models \operatorname{DC} + \operatorname{EP}+\operatorname{NIP}$, we prove that its open core is interdefinable with its o-minimal reduct $(\mathbb{R},<,+,0,1)$. 
Since the graph of $x\mapsto cx$ is dense for $c\in \Q(t)\setminus \Q$, the theory of $\calR_t$ provides a negative answer to the following question.

\begin{question3}[{\cite[p. 705]{DMS2}}] Let $T$ be a complete o-minimal extension of the theory of densely ordered groups. If $\tilde{T}$ is any theory (in any language) having $T$ as an open core, and some model of $\tilde{T}$ defines a somewhere dense graph, must $\operatorname{EP}$ fail for $\tilde{T}$?
\end{question3}

\noindent Our counterexample  $\calR_t$ does not expand a field and we don't know whether Question 2 (or Question 3) has a positive answer if we require $\mathcal{R}$ (or $T$) to expand an ordered field.\newline

\noindent We say that $\mathcal R$ has \textbf{definable Skolem functions} if for every definable set $A\subseteq R^m \times R^n$ there is a definable function $f: R^m \to R^n$ such that $(a,f(a)) \in A$ whenever $a \in A$ and there exists $b\in \R^n$ with $(a,b) \in A$. Every o-minimal expansion of an ordered group with a distinguished positive element has definable Skolem functions, but all documented examples of non-o-minimal structures with o-minimal core do not.

\begin{question4}[{\cite[p. 1409]{DMS}}] If $\mathcal{R}$ has definable Skolem functions and $\mathcal{R}^{\circ}$ is o-minimal, is $\mathcal{R}$ o-minimal?
\end{question4}

\noindent The answer is again negative. We say $\calR$ satisfies \textbf{uniform finitness} (short: $\calR\models \operatorname{UF}$) if for every $m,n\in \N$ and every $A\subseteq R^{m+n}$ definable in $\calR$ there exists $N\in \N$ such that for every $a\in R^m$ the set $\{ b \in R^n \ : \ (a,b) \in A\}$ is either infinite or contains at most $N$ elements. By \cite[Theorem A]{DMS}, if $\calR\models \operatorname{DC}+\operatorname{UF}$ and expands an ordered group, then $\calR^{\circ}$ is o-minimal. Using a construction due to Winkler \cite{W75} and following a strategy of Kruckman and Ramsey \cite{KR18}, we establish that if $\mathcal R \models \operatorname{DC} + \operatorname{UF}$, then $\calR$ has an expansion $\mathcal S$ such that $\mathcal S$ has definable Skolem functions and satisfies $\mathcal S^{\circ} \eqdf \mathcal{R}^{\circ}$. Thus if $\calR$ also expands an ordered group, then $\mathcal R^{\circ}$ is o-minimal and so is $\mathcal S^{\circ}$.

\subsection*{Acknowledgements} A.B.G. was  supported by the National Science Foundation Graduate Research Fellowship Program under Grant No. DGE -- 1746047.  P.H. was partially supported by NSF grant DMS-1654725. We thank Pantelis Eleftheriou for helpful conversations on the topic of this paper, in particular for bringing Question 1 to the attention of P.H. at a conference at the Bedlewo Conference Center. We thank Chris Miller for helpful comments on a draft of this paper.

\subsection*{Notation}  We will use $m, n$ for natural numbers and $\kappa$ for a cardinal. 
Let $\calL$ be a language and $T$ an $\calL$-theory. Let $\mathcal{M}\models T$. We follow the usual convention to denote the universe of $\calM$ by $M$. In this situation, $\calL$-definable means $\calL$-definable with parameters. Let $b$ be a tuple of elements in $M$, and let $A\subseteq M$. We write $\tp_{\calL}(b|A)$ for the $\calL$-type of $b$ over $A$.
If $\calN$ is another model of $T$ and $h$ is an embedding of a substructure of $\mathcal{M}$ containing $A$ into $\calN$, then $h \tp(b|A)$ is the type containing all formulas of the form $\varphi(x,h(a))$ where $\varphi(x,a)\in \tp(b|A)$.

\section{Question \ref{ques:unary}}
Let $\overline{\mathbb{R}}$ be the real ordered field $(\mathbb{R},<,+,\cdot)$ and let $\mathbb{R}_{\exp}$ be the expansion of the real field by the exponential function $\exp$. Let $I\subseteq \mathbb{R}$ be a dense $\operatorname{dcl}_{\mathbb{R}_{\exp}}$-independent set.  
Let $\tau\in I$ be such that $\tau > 1$. Set 
\[
J := \bigcup_{a \in I\setminus\{\tau\}} \{|a|,|a|^{\tau},|a|+|a|^{\tau}\}.
\]
By \cite[2.25]{DMS2} the open core of $(\mathbb{R}_{\exp},I)^{\circ}$ is interdefinable with $\mathbb{R}_{\exp}$. Since $(\overline{\mathbb{R}},J)$ is a reduct of  $(\mathbb{R}_{\exp},I)$, we have that $(\overline{\mathbb{R}},J)^{\circ}$ is a reduct of $(\mathbb{R}_{\exp},I)^{\circ}$. As the latter structure is o-minimal, we have that $(\overline{\mathbb{R}},J)^{\circ}$ is o-minimal as well. In order to show that Question \ref{ques:unary} has a negative answer, it is left to show that $(\overline{\mathbb{R}},J)^{\circ}$ defines a set not definable in $\overline{\mathbb{R}}$. Since $\overline{\mathbb{R}}$ only defines raising to rational powers, it suffices to prove the definability of $x\mapsto x^{\tau}$ on an unbounded interval.

\begin{lemma}\label{lem:pow} Let $u_1,u_2,u_3 \in J$ such that $1<u_1<u_2$ and $u_1+u_2=u_3$. Then there is $a\in I\setminus \{\tau\}$ such that $u_1=|a|$ and $u_2=|a|^{\tau}$. 
\end{lemma}
\begin{proof}
For $a\in I\setminus \{\tau\}$  observe that $|a|$, $|a|^{\tau}$ and $|a|+|a|^{\tau}$ are interdefinable in $\mathbb{R}_{\exp}$ over $\tau$.
Since $u_1+u_2=u_3$, we have $u_1,u_2,u_3$ are $\operatorname{dcl}_{\mathbb{R}_{\exp}}$-dependent. Because $I$ is $\dcl_{\R_{exp}}$-independent, there are $a\in I\setminus \{\tau\}$ and $i,j\in\{1,2,3\}$ such that 
\[
u_i,u_j \in \{|a|, |a|^{\tau},|a|+|a|^{\tau}\}. 
\]
Let $\ell \in \{1,2,3\}$ such that $\ell \neq i$ and $\ell \neq j$. Note $u_{\ell}$ is $\operatorname{dcl}_{\mathbb{R}_{\exp}}$-dependent over $u_i$ and $u_j$. Thus $u_{\ell}\in \{|a|, |a|^{\tau},|a|+|a|^{\tau}\}.$ Since $|a|>0$, we obtain from $u_1+u_2=u_3$ that
\[
u_1,u_2 \in \{|a|,|a|^{\tau}\}.
\]
Since $1<u_1<u_2$ and $\tau>0$, we have that $u_1=|a|$ and $u_2=|a|^{\tau}$.
\end{proof}

\begin{prop}
The graph of $x\mapsto x^{\tau}$ on $\mathbb{R}_{\geq 1}$ is definable in $(\overline{\mathbb{R}},J)^{\circ}$.
\end{prop}
\begin{proof}
By Lemma \ref{lem:pow} the structure $(\overline{\mathbb{R}},J)$ defines 
\[
\{ (|a|,|a|^{\tau})  \ : |a|>1, \ a \in I\setminus \{\tau\} \}.
\]
Since $I$ is dense in $\R$, the closure of this set is the graph of $x\mapsto x^{\tau}$ on $\mathbb{R}_{\geq 1}$, and hence definable in $(\overline{\mathbb{R}},J)^{\circ}$.
\end{proof}
\noindent We conclude that $(\overline{\mathbb{R}},J)^{\circ}$ is a proper expansion of $\overline{\mathbb{R}}$.

\section{Questions 2 and 3}

In this section we give negative answers to Questions 2 and 3. Let $\mathbb{Q}(t)$ be the field of rational functions in the variable $t$. 
We expand $(\R,<,+,0,1)$ to a $\Q(t)$-vector space such that for each non-constant $q(t)\in \Q(t)$ the graph of multiplication by $q(t)$ is dense.\newline

\noindent We now construct such a $\Q(t)$-vector space structure on $\R$. Let $\mathbf{1}$ be the multiplicative identity of $\Q(t)$.
We fix a dense basis $\calB$ of $\R$ as a $\Q$-vector space, and a basis  $I$ of  $\Q(t)$ as a $\Q$-vector space such that $\mathbf{1} \in I$.
We choose a sequence of functions $\{\tilde{f_{\gamma}}: I \to \calB\}_{\gamma \in 2^{\aleph_0}}$ such that
\[ 
\calB = \displaystyle\bigcup_{\gamma \in 2^{\aleph_0}} \tilde{f_{\gamma}}(I)
\]
and for all $\gamma \in 2^{\aleph_0}$:
    \begin{itemize}
        \item $\tilde{f_{\gamma}}$ is injective,
            \item for all $\eta \in 2^{\aleph_0}$ with $\eta \neq \gamma$, $\tilde{f_{\eta}}(I) \cap \tilde{f_{\gamma}}(I) = \emptyset$,
        \item for all open intervals $J_1,\ldots ,J_m\subseteq \R$ open intervals and all pairwise distinct $p_1(t),\ldots ,p_m(t)\in I$ there exists $\gamma \in 2^{\aleph_0}$ such that \[\tilde{f_{\gamma}}(p_1(t)) \in J_1, \ldots ,\tilde{f_{\gamma}}(p_m(t))\in J_m.\] 
    \end{itemize}

\noindent Since the order topology on the real line has a countable base, it is easy to see that such a sequence of functions exists. For each $\gamma \in 2^{\aleph_0}$, $\tilde{f_{\gamma}}$ is defined on the basis $I$ of $\mbb{Q}(t)$. Therefore, we can extend each $\tilde{f_{\gamma}}: I \to \mc{B}$ to a $\mbb{Q}$-linear map $f_{\gamma}: \mbb{Q}(t) \to \mbb{R}$.

\begin{lemma}\label{lem:fgamma}
Let $a\in \R$. Then there are unique $\gamma_1,\dots, \gamma_n \in 2^{\aleph_0}$ and $p_1(t),\dots,p_n(t) \in \Q(t)$ such that 
\[
a = f_{\gamma_1}(p_1(t)) + \dots + f_{\gamma_n}(p_n(t)).
\]
\end{lemma}
\begin{proof}
Since $\calB$ is a basis of $\R$ as a $\Q$-vector space, there are unique $b_1,\dots,b_n\in \mathcal{B}$ and $u_1,\dots,u_n\in \Q$ such that $a=u_1b_1 + \dots +u_n b_n.$ By the above construction, there are unique $\gamma_1,\dots,\gamma_n\in 2^{\aleph_0}$ and $q_1(t),\dots,q_n(t)\in I$ such that $b_i=f_{\gamma_i}(q_i(t))$ for $i=1,\dots,n$.
Then  by $\Q$-linearity of the $f_{\gamma_i}$'s
\begin{align*}
a &= u_1b_1 + \dots + u_n b_n\\
&= u_1f_{\gamma_1}(q_1(t)) + \dots + u_n f_{\gamma_n}(q_n(t))\\
&=f_{\gamma_1}(u_1q_1(t)) + \dots + f_{\gamma_n}(u_nq_n(t)).
\end{align*}
Set $p_i := u_iq_i(t)$ for $i=1,\dots,n$.
\end{proof}

\noindent 
We now introduce a $\Q$-linear map $\lambda: \Q(t) \times \R \to \R$.
Let $q(t)\in \Q(t)$ and $a \in \R$. 
By Lemma \ref{lem:fgamma} there are unique $\gamma_1,\dots, \gamma_n \in 2^{\aleph_0}$ and $p_1(t),\dots,p_n(t) \in \Q(t)$ such that 
\[
a = f_{\gamma_1}(p_1(t)) + \dots + f_{\gamma_n}(p_n(t)).
\]
We define
\[
\lambda(q(t),a) := f_{\gamma_1}(q(t)\cdot p_1(t)) + \dots + f_{\gamma_n}(q(t) \cdot p_n(t)).
\]
By Lemma \ref{lem:fgamma}, the function $\lambda$ is well-defined. For $q(t) \in \mbb{Q}(t)$, we write $\lambda_{q(t)}$ for the map taking $a \in \R$ to $\lambda(q(t),a).$

\begin{prop}\label{prop:qtvs}
The additive group $(\R,+)$ with $\lambda$ as scalar multiplication is an $\Q(t)$-vector space.
\end{prop}
\begin{proof}
We only verify the following vector spaces axioms: for all $a \in \R$ and for all $q_1(t),q_2(t) \in \Q(t)$.
\[
\lambda_{q_1(t) \cdot q_2(t)}(a) = \lambda_{q_1(t)}\big( \lambda_{q_2(t)}(a)\big).
\]
The other axioms can be checked using similar arguments and we leave this to the reader.\newline

\noindent Let $a \in \R$ and let $q_1(t),q_2(t) \in \Q(t)$.
By Lemma \ref{lem:fgamma} there are unique $\gamma_1,\dots, \gamma_n \in 2^{\aleph_0}$ and $p_1(t),\dots,p_n(t) \in \Q(t)$ such that 
\[
a = f_{\gamma_1}(p_1(t)) + \dots + f_{\gamma_n}(p_n(t)).
\]
We obtain
\begin{align*}
\lambda_{q_1(t)}(\lambda_{q_2(t)}(a)) &=\lambda_{q_1(t)}\big(\lambda_{q_2(t)}\big( \sum_{i=1}^n f_{\gamma_i}(p_i(t))\big)\\
&= 
\lambda_{q_1(t)}(\sum_{i=1}^n f_{\gamma_i}(q_2(t)\cdot p_i(t)))\\
&=\sum_{i=1}^n f_{\gamma_i}(q_1(t) \cdot (q_2(t)\cdot p_i(t)))\\
&=\lambda_{q_1(t)\cdot q_2(t)}\big(\sum_{i=1}^n f_{\gamma_i}(p_i(t))\big)\\
&=\lambda_{q_1(t) \cdot q_2(t)}(a).
\end{align*}

\end{proof}

\noindent Let $\calL$ be the language of $(\R,<,+,0,1)$, and let $T$ be its theory; that is the theory of ordered divisible abelian groups with a distinguished positive element. It is well-known that $T$ has quantifier-elimination and is o-minimal. We will use various consequences of this fact throughout this section. Most noteworthy: when $\calM\models T$, $X\subseteq M^n$ is $\calL$-definable over $A$ and there is $b=(b_1,\dots,b_n)\in X$ such that $b_1,\dots,b_n$ are $\Q$-linearly independent over $A$, then $X$ has interior.\newline

\noindent Let $\calR_{t}=(\R,<,+,0,1,(\lambda_{q(t)})_{q(t)\in \mathbb{Q}(t)})$ be the expansion of $(\R,<,+,0,1)$ by function symbols for $\lambda_{q(t)}$ where $q(t)\in \Q(t)$. We denote the language of $\calR_t$ by $\calL_{t}$.

\subsection{Density}

Let $p=(p_1(t),\dots,p_n(t))\in \Q(t)^n$. Let $\lambda_p: \R \to \R^n$ denote the function from $\R$ that maps $a$ to $(\lambda_{p_1(t)}(a),\dots \lambda_{p_n(t)}(a))$. The main goal of this subsection is to show the density of the image of $\lambda_p$ when the coordinates of $p$ are $\Q$-linearly independent.
\begin{lemma}\label{lem:idensitiy} Let $p=(p_1(t), \ldots, p_n(t)) \in I$ be such that $p_i(t) \neq p_j(t)$ for $i\neq j$. Then the image of $\lambda_{p}$ is dense in $\R^{n}$.
\end{lemma}

\begin{proof}
Let $J_1, \ldots, J_n$ be open intervals in $\mbb{R}$. Since $p_1(t), \ldots, p_n(t)$ are distinct elements of $I$, there exists $\gamma \in 2^{\aleph_0}$ such that
\[
f_{\gamma}(p_1(t)) \in J_1, \ldots, f_{\gamma}(p_n(t)) \in J_n.
\]
For each $i \in \{1, \ldots, n\}$, we have 
\[
\lambda_{p_i(t)}(f_{\gamma}(\mathbf{1})) = f_{\gamma}(p_i(t))
\]
by definition of $\lambda_{p_i(t)}$. Therefore,
\[
\big(\lambda_{p_1(t)}(f_{\gamma}(\mathbf{1})),\ldots, \lambda_{p_n(t)}(f_{\gamma}(\mathbf{1}))\big) \in J_1 \times J_2 \times \ldots \times J_n.
\]
\end{proof}

\begin{prop}\label{prop:density}
Let $q=(q_1(t),\dots,q_m(t))\in \Q(t)^m$ be such that $q_1(t),\dots,q_m(t)$ are $\Q$-linearly independent. Then the image of $\lambda_q$ is dense in $\R^n$.
\end{prop}
\begin{proof} Let $n \in \N$, let $p_1(t),\ldots,p_n(t) \in I$ be distinct non-constant, and let $A=(u_{i,j})_{i=1,\dots,m,j=0,\dots,n}$ be an $m \times (n+1)$ matrix with rational entries such that
\begin{align*}
    q_1(t) &= u_{1,0} \mathbf{1} + u_{1,1} p_1(t) + \ldots + u_{1,n}p_n(t) \\
    q_2(t) &= u_{2,0} \mathbf{1} + u_{2,1} p_1(t) + \ldots + u_{2,n}p_n(t)\\
    &\vdots \\
    q_m(t) &= u_{m,0} \mathbf{1} + u_{m,1} p_1(t) + \ldots + u_{m,n}p_n(t).
\end{align*}
By definition of $\lambda_{\mathbf{1}}, \lambda_{p_1(t)}, \ldots, \lambda_{p_n(t)}$, we have for each $i \in \{1, \ldots, m\}$
\[
\lambda_{q_i(t)}(x) = u_{i,0} \lambda_{\mathbf{1}}(x) + u_{i,1} \lambda_{p_1(t)}(x) + \ldots + u_{i,n} \lambda_{p_n(t)}(x).
\]
Therefore,
\[
A \lambda_{(\mathbf{1},p_1(t),\ldots,p_n(t))} = \lambda_{(q_1(t),\ldots,q_m(t))}.
\]
Since $q_1(t),\ldots,q_m(t)$ are $\Q$-linearly independent, the matrix $A$ has rank $m$. Thus multiplication by $A$ is a surjective map from $\R^n$ to $\R^m$. Since matrix multiplication is continuous and continuous surjections preserve density, the image of $A \lambda_{(\mathbf{1},p_1(t),\ldots,p_n(t))}$ is dense in $\R^m$ by Lemma \ref{lem:idensitiy}.
\end{proof}

\subsection{Axiomatization and QE}
In this subsection, we will find an axiomatization of $\calR_t$ and show that this theory has quantifier elimination. Indeed, we will prove that the following subtheory of the $\calL_t$-theory of $\calR_t$ already has quantifier-elimination. 

\begin{defn} \label{axioms}
Let $T_t$ be the $\calL_t$-theory extending $T$ by axiom schemata stating that for every model $\calM=(M,<,+,0,1,(\lambda_{q(t)})_{q(t)\in \Q(t)})\models T_t$
\begin{itemize}
    \item[\emph{(T1)}] $(M,+,0,(\lambda_{q(t)})_{q(t)\in \Q(t)})$ is a $\Q(t)$-vector space.
    \item[\emph{(T2)}] If $q_1(t), \ldots ,q_m(t) \in \Q(t)$ are $\Q$-linearly independent, then the image of the  $\lambda_{(q_1(t),\dots, q_m(t))}$ is dense in $M^m$.
\end{itemize}
\end{defn}

\noindent By Proposition \ref{prop:qtvs} and Proposition \ref{prop:density} we know that $\calR_t\models T_t$. Let $\calM\models T_t$. We observe that by (T1) the $\calL_t$-substructures of $\calM$ are precisely the $\Q(t)$-linear subspaces of $\calM$ containing $1$.

\begin{lemma}\label{lem:indep} Let $\calM\models T_t$ and let $\calA$ be an $\calL_t$-substructure of $\calM$. Let $b\in M\setminus A$ and let $p_1(t),\dots, p_n(t)\in \Q(t)$ be $\Q$-linearly independent. Then $\lambda_{p_1(t)}(b),\dots, \lambda_{p_n(t)}(b)$ are $\Q$-linearly independent over $A$.
\end{lemma}
\begin{proof}
Since $\calA$ is a $\Q(t)$-linear subspace of $\calM$, we know that $\lambda_{q(t)}(b)\notin A$ for all non-zero $q(t)\in \Q(t)$. Let $u_1,\dots,u_n\in \Q$. Since $\calM$ is a $\Q(t)$-vector space,
\[
u_1 \lambda_{p_1(t)}(b) + \dots + u_n \lambda_{p_n(t)}(b) = \lambda_{u_1p_1(t)+ \dots + u_n p_n(t)}(b).
\] 
Because $p_1(t),\dots, p_n(t)\in \Q(t)$ are $\Q$-linearly independent and $b\notin A$, we get that 
\[
u_1 \lambda_{p_1(t)}(b) + \dots + u_n \lambda_{p_n(t)}(b) \in A \Rightarrow u_1=\dots = u_n = 0.  
\]
Thus $\lambda_{p_1(t)}(b),\dots, \lambda_{p_n(t)}(b)$ are $\Q$-linearly independent of $A$.
\end{proof}

\begin{prop}\label{prop:qe}
The theory $T_{t}$ has quantifier elimination.
\end{prop}
\begin{proof}
Let $\calM, \calN \models T_{t}$ be such that $\calN$ is $|\calM|^+$-saturated. Let $\calA \subseteq \calM$ be a substructure that embeds into $\calN$ via the embedding $h: \calA \into \calN$. Let $b\in M\setminus A$. To prove quantifier elimination of $T_t$, it is enough to show that the embedding $h$ extends to an embedding of the $\calL_t$-substructure generated by $\calA$ and $b$.\newline

\noindent Consider the tuple $(\lambda_{p(t)}(b))_{p(t) \in I}$.  We first find $c\in N$ such that
\begin{equation}\label{eq:qetypeeq}
h\tp_{\calL}\big( (\lambda_{p(t)}(b))_{i \in I} | A\big) = tp_{\calL}\big( (\lambda_{p(t)}(c))_{i \in I} | h(A)\big)
\end{equation}
By saturation of $\calN$ it is enough to show that for every 
\begin{itemize}
    \item pairwise distinct $p_1(t),\dots,p_n(t)\in I$,
    \item $\calL$-formula $\psi(x,y)$ and $a \in A^{|y|}$ such that 
    \[
    \calM \models \psi\big(\lambda_{p_1(t)}(b),\dots, \lambda_{p_n(t)}(b),a\big),
    \]
\end{itemize}
there is $c\in N\setminus h(A)$ such that 
\[ 
\calN \models \psi\big(\lambda_{p_1(t)}(c),\dots, \lambda_{p_n(t)}(c),h(a)\big).
\]
Because $I$ is a $\Q$-linear basis of $\Q(t)$, the sequence $(\lambda_{p(t)}(b))_{i \in I}$ is $\Q$-linear independent over $A$ by Lemma \ref{lem:indep}. Thus
the set 
\[
\{ d \in N^n \ : \ \calN \models \psi(d,h(a))\}
\]
has interior. The existence of $c$ now follows from (T2) and saturation of $\mathcal{N}$.\newline

\noindent Let $c\in N$ be such that \eqref{eq:qetypeeq} holds. Let $\mathcal{X}$ be the $\Q$-linear subspace of $\calM$ generated by $(\lambda_{p(t)}(b))_{i \in I}$ and $\calA$. Let $\mathcal{Y}$  be the $\Q$-linear subspace of $\calN$ generated by $(\lambda_{p(t)}(c))_{i \in I}$ and $h(\calA)$. Observe that $\mathcal{X}$ is the $\Q(t)$-subspace of $\calM$ generated by $b$ and $\calA$, and $\mathcal{Y}$ is the $\Q(t)$-subspace of $\calN$ generated by $c$ and $h(\calA)$. Hence $\mathcal{X}$ and $\mathcal{Y}$ are $\calL_t$-structures 
of $\calM$ and $\calN$ respectively. Since $c$ satisfies \eqref{eq:qetypeeq}, there is an $\calL$-isomorphism $h': \mathcal{X} \to \calC$ extending $h$ and mapping $\lambda_{p(t)}(b)$ to $\lambda_{p(t)}(c)$ for each $p(t)\in I$. It follows easily that this $h'$ is $\Q(t)$-linear and hence an $\calL_t$-isomorphism extending $h$.
\end{proof}
\begin{cor}\label{cor:type}
Let $\calM, \calN \models T_t$, let $\calA\subseteq \calM$ be an $\calL_t$-substructure such that $h: A \into \calN$ is an $\calL_t$-embedding. Let $b\in M\setminus A$ and $c\in N\setminus h(A)$ such that
\[
h\tp_{\calL}(\big(\lambda_{p(t)}(b)\big)_{p(t)\in I}|A) = \tp_{\calL}(\big(\lambda_{p(t)}(c)\big)_{p(t)\in I}|h(A)).
\]
Then $\tp_{\calL_t}(b|A) = \tp_{\calL_t}(c|h(A))$.
\end{cor}
\begin{proof}
Let $\mathcal{X}$ be the $\Q$-linear subspace of $\calM$ generated by $(\lambda_{p(t)}(b))_{i \in I}$ and $\calA$, and let $\mathcal{Y}$  be the $\Q$-linear subspace of $\calN$ generated by $(\lambda_{p(t)}(c))_{i \in I}$ and $h(\calA)$. It is easy to check that $\mathcal{X}$ and $\mathcal{Y}$ are $\calL_t$-substructures of $\calM$ and $\calN$ respectively. By our assumption on $b$ and $c$, the embedding $h$ extends on an $\mathcal{L}$-isomorphism $h'$ between $\mathcal{X}$ and $\mathcal{Y}$ mapping $\lambda_{p(t)}(b)$ to $\lambda_{p(t)}(c)$ for each $p(t)\in I$. Since $h$ is an $\calL_t$-embedding, it follows easily that $h'$ is an $\calL_t$-isomorphism between $\mathcal{X}$ to $\mathcal{Y}$. Since $T_t$ has quantifier elimination, we get that $\tp_{\calL_t}(b|A) = \tp_{\calL_t}(c|h(A))$.
\end{proof}

\begin{prop}
The theory of $\calR_t$ is axiomatized by the theory $T_t$ in conjunction with the axiom scheme that specifies $\tp_{\calL}(\big(\lambda_{p(t)}(1)\big)_{p(t)\in I})$.
\end{prop}

\begin{proof}
Let $T_t^*$ be the theory described in the statement. Since $\calR_t\models T_t$, we immediately get that $\calR_t \models T_t^*$. It is left to show that $T_t^*$ is complete. Let $\calM$ and $\calN$ be model of $T_t^*$ of size $\kappa >\aleph_0$. By Corollary \ref{cor:type}, $1_{\calM}$ and $1_{\calN}$ satisfy the same $\calL_t$-type.
Thus there is an $\calL_t$-isomorphism $h$ mapping the $\calL_t$-substructure of $\calM$ generated by $1_{\calM}$ to the $\calL_t$-substructure of $\calN$ generated by $1_{\calN}$. By the proof of Proposition \ref{prop:qe} this isomorphism $h$ extends to an $\calL_t$-isomorphism between $\calM$ and $\calN$. 
\end{proof}

\subsection{Exchange property}

In this subsection we establish that every model of $T_t$ has the exchange property. We will do so by showing that the definable closure in such a model is equal to the $\Q(t)$-linear span.

\begin{lemma}
\label{cor:defclosed}
Let $\mathcal{M} \models T_t$ and let $\calA\subseteq \calM$ be an $\calL_t$-substructure. Then $\calA$ is definably closed.
\end{lemma}
\begin{proof}
Without loss of generality, we can assume that $\mathcal{M}$ is $|A|^+$-saturated. 
Let $b\in M\setminus A$. It is enough to show that there exists $c\in M$ such that $b\neq c$ and $\tp_{\calL_t}(b|A) = \tp_{\calL_t}(c|A)$.  
By Corollary \ref{cor:type} it is sufficient to find $c\in M$ such that $b\neq c$ and
\[
\tp_{\calL}(\big(\lambda_{p(t)}(b)\big)_{p(t)\in I}|A) = \tp_{\calL}(\big(\lambda_{p(t)}(c)\big)_{p(t)\in I}|A).
\]
Let $\varphi(x,y)$ be an $\calL$-formula, $p_1(t),\dots,p_m(t)\in I$ and $a\in A^n$ such that
\[
\calM \models \varphi(\lambda_{p_1(t)}(b),\dots,\lambda_{p_m(t)}(b),a).
\]
By saturation of $\calM$, we only need to find $c\in M$ such that $c\neq b$ and $\calM \models
\varphi(\lambda_{p_1(t)}(b),\dots,\lambda_{p_m(t)}(b),a)$. 
By Lemma \ref{lem:indep}, $\big(\lambda_{p(t)}(b)\big)_{p(t)\in I}$ is $\Q$-linear independent over $A$. Thus the set
\[
X := \{ d \in M^m \ : \ \calM \models \varphi(d,a)\}
\]
has interior. By axiom (T2) the intersection  
\[
\{ (\lambda_{p_1(t)}(c),\dots,\lambda_{p_m(t)}(c)) \ : \ c \in M\}\cap X
\]
is dense in $X$.
\end{proof}

\begin{cor}\label{cor:defvec}
Let $\mathcal{M}\models T_t$ and let $Z\subseteq M$. Then
the $\calL_t$-definable closure of $Z$ is the $\Q(t)$-subspace of $\calM$ generated by $Z$.
\end{cor}
\begin{proof}
By Lemma \ref{cor:defclosed} the definable closure of $Z$ is the $\calL_t$-substructure generated by $Z$. However, the latter is just the $\Q(t)$-subspace of $\calM$ generated by $Z$.
\end{proof}

\noindent The exchange property for $T_t$ follows immediately from Corollary \ref{cor:defvec} and the classical Steinitz exchange lemma for vector spaces.

\begin{prop}\label{cor:EP}
The theory $T_t$ has the exchange property.
\end{prop}



\subsection{Open core}

Let $\calM\models T_t$. Then by Axiom (T2) it defines functions from $M$ to $M$ whose graph is dense. We already know that $\calM$ has EP by Proposition \ref{cor:EP}. To give a negative answer to Question 3, it is left to show that every open subset of $M^n$ definable in $\calM$ is already definable in the reduct $(M,<,+,0,1)$.

\begin{thm}\label{thm:opencorevs}
The theory $T$ is an open core of $T_t$.
\end{thm}
\begin{proof}
Let $\mathcal{M}\models T_t$. We prove that every open set is $\calL$-definable. Without loss of generality, we can assume that $\mathcal{M}$ is $\aleph_0$-saturated.
Let $C$ be a finite subset. Using Boxall and Hieronymi \cite[Corollary 3.1]{BH12}, we will show that for every $n\in \N$ and every subset of $M^n$ that is $\calL_t$-definable over $C$, is also $\calL$-definable over $C$. Let $n\in \N$ and $p_1(t),\dots,p_{n-1}(t)\in I$ be distinct and non-constant. We define
\[
D := \{ \big(a,\lambda_{p_1(t)}(a),\dots,\lambda_{p_{n-1}(t)}(a)\big) \ : \ a \notin \dcl_{\calL_t}(C)\}.
\]
From (T2) and saturation of $\calM$, it follows easily that $D$ is dense in $M^n$. Thus Condition (1) of \cite[Corollary 3.1]{BH12} is satisfied. Condition (3) of \cite[Corollary 3.1]{BH12} holds by Corollary \ref{cor:type}. It is only left to establish Condition (2).\newline

\noindent Let $b\in D$ and $a\notin \dcl_{\calL_t}(C)$ be such that $b=\big(a,\lambda_{p_1(t)}(a),\dots,\lambda_{p_{n-1}(t)}(a)\big)$. Let $U\subseteq M^n$ be open and suppose that $\tp_{\calL}(b|C)$ is realized in $U$. We need to show that $\tp_{\calL}(b|C)$ is realized in $U\cap D$.
By Lemma \ref{lem:indep} we know that the coordinates of $b$ are $\Q$-linearly independent over $\dcl_{\calL_t}(C)$. Thus the set of realizations of $\tp_{\calL}(b|C)$ is open, and so is its intersection with the open set $U$. Denote this intersection by $V$. By (T2) and $\aleph_0$-saturation of $\calM$, we find $a'\notin \dcl_{\calL_t}(C)$ such that $b'=\big(a',\lambda_{p_1(t)}(a'),\dots,\lambda_{p_{n-1}(t)}(a')\big)$. Now $b'$ is the desired realization of $\tp_{\calL}(b|C)$ in $U\cap D$.
\end{proof}

\noindent By Theorem \ref{thm:opencorevs} every model of $T_t$ has o-minimal open core and thus is definably complete.

\subsection{Neostability results}
We will now show that $T_t$ is NIP, but not strong.
We use an equivalent definition of the independence property in the theorem below, namely that in a monster model $\calM$ of $T_t$ there is no formula $\varphi(x, y)$ and no element $a \in M$ such that for some indiscernible sequence $(b_i)_{i<\omega}$ of tuples in $M^{|y|}$ we have 
$$\calM \models \varphi(a, b_i) \hbox{ if and only if }  i<\omega \hbox{ is even}.$$
For a proof that this is equivalent to the classic definition of NIP, see \cite{S15}.

\begin{thm}
Every completion of the theory $T_t$ has \emph{NIP}.
\end{thm}

\begin{proof}
We let $\calM \models T_{t}$ be a monster model of $T_t$.
We suppose for a contradiction that there is an $\calL_t$-formula $\varphi(x, y)$ along with an element $a \in M$ and indiscernible sequence $( b_i )_{i< \omega}$ of elements in $M^{|y|}$ that witnesses IP,  i.e. $\calM \models \varphi(a, b_i)$ precisely if $i$ is even. Let $|y| = n$ and for each $i<\omega$ we denote the $j$-coordinate of $b_i$ by $b_{i,j}$. By quantifier elimination in the language $\calL_t$, we can assume that the formula $\varphi(a, y)$ is equal to a boolean combination of formulas of the form
\begin{enumerate}
    \item $a- \sum_{j=1}^n \lambda_{q_{j}(t)}(y_j) =0$
\item $a - \sum_{j=1}^{n} \lambda_{q_{j}(t)}(y_j) >0$
\end{enumerate}
where $q_1,\dots,q_{n}(t)\in \Q(t)$.
Since NIP is preserved under boolean combinations, we can assume that $\varphi$ is of the form (1) or (2). For ease of notation, let $f(b_i)=\sum_{j=1}^{n} \lambda_{q_{j}(t)}(b_{i,j})$ for each $i<\omega$. \newline

\noindent Suppose that $\varphi$ is of form (1). 
We suppose without loss of generality that 
$
a - f(b_i) =0
$
holds if and only if $i<\omega$ is odd. Then we have that $f(b_1)=f(b_3)$, but $f(b_2)\neq f(b_1)$. Thus we conclude $\tp_{\calL_t}(b_1b_2) \neq \tp_{\calL_t}(b_2b_3)$, contradicting indiscernability.\newline


\noindent Now assume that $\varphi$ is of the form (2).
Without loss of generality assume that $
a - f(b_i)>0$
holds if and only if $i<\omega$ is odd.
Then for all $i<\omega$, we have $a < f(b_{2i})$  and $a> f(b_{2i+1}).$
However, this means that
$f(b_1) < f(b_2)$ and $f(b_2)>f(b_3)$. 
So we again obtain $\tp_{\calL_t}(b_1b_2) \neq \tp_{\calL_t}(b_2b_3)$, contradicting indiscernability.

\end{proof}

\noindent Thus $\calR_t\models \operatorname{DC}+\operatorname{EP}+\operatorname{NIP}$, but $\calR_t$ is not o-minimal. This gives a negative answer to Question 2.

\begin{prop}
No completion of the theory $T_t$ is strong.
\end{prop}
\begin{proof}
Fix a family $(q_j(t))_{j\in \N}$ of distinct elements of $I$.
Consider the family of $\calL_t$-formulas given by $(\lambda_{q_j(t)}(x) \in (a_k,b_k))_{j,k \in \N}$ such that 
\begin{itemize}
    \item $(a_k, b_k) \cap (a_{\ell},b_{\ell}) = \emptyset$,
for all $\ell \neq k \in \N$, and
\item the tuples $(a_k,b_k)_{k \in \N}$ form an indiscernible sequence.
\end{itemize}
In the array that corresponds to varying $j \in 
\N$ along the rows and the $k\in \N$ along the columns, it is easy to see that formulas in the same row are pairwise inconsistent.
However, for every path $(\lambda_{q_{\gamma(k)}(t)}(x) \in (a_{\gamma(k)}, b_{\gamma(k)}))_{k \in \N}$, every finite subset of these formulas are consistent by our axiom scheme (T2). 
So by compactness, every path through the entire array is consistent.
\end{proof}

\section{Question 4}

Let $T$ be a theory extending the theory of dense linear orders without endpoints in an language $\calL$. We write $T\models \operatorname{UF}$ if every model of $T$ satisfies $\operatorname{UF}$. The main goal of this section is to establish the following theorem. 

\begin{thm}\label{thm:sklmf}
Suppose that $T\models \operatorname{DC}+\operatorname{UF}$. Let $T'$ be an open core of $T$. 
There is a theory $T_{\operatorname{Sk}}^{\infty}$ extending $T$ such that $T_{\operatorname{Sk}}^{\infty}$ has  definable Skolem functions and $T'$ is an open core of $T_{\operatorname{Sk}}^{\infty}.$
\end{thm}
\noindent This immediately gives a negative answer to Question 4, as there are many documented examples of a theory $T$ with $T \models \operatorname{DC}+\operatorname{UF}$ and  o-mininal open core that is not o-minimal itself. To prove Theorem \ref{thm:sklmf} we follow a strategy of Kruckman and Ramsey \cite{KR18} and rely on a construction due to Winkler \cite{W75} allowing us to successively add definable Skolem functions to the language $\mathcal{L}$ of a given theory $T$ while preserving uniform finiteness. As explained below, this construction preserves the open definable sets by a result from \cite{BH12}. We begin by recalling notations and results from \cite{W75}.

\subsection{Skolem expansions}

Let $\calL$ be a language and $\Theta = \{\theta_{t}(x,y): t < |\calL|\}$ be an enumeration of all $\calL$-formulas $\varphi(x,y)$ where the variable $y$ has length 1. Define $\calL_{Sk}$ to be $\calL \cup \{f_{t}: t < |\calL|\}$, where the arity of $f_{t}$ is the length of the tuple $x$ appearing in $\theta_{t}(x,y)$.\newline

\noindent The \textbf{Skolem expansion $T_{+}$ of $T$} is the $\calL_{Sk}$-theory
\[
T_+ = T \cup \{ \forall x \exists y\left ( \theta_{t}(x,y) \rightarrow \theta_{t}(x, f_{t}(x))\right) \ : \ t<|\calL|\}.
\]
We refer to the $f_{t}$'s as Skolem functions.\newline

\noindent \emph{From here on we assume that $T$ has quantifier elimination in the language $\calL$ and assume that for each $\calL$-definable function $f$ there is an $\calL$-term $t$ such that $T\models \forall x \ f(x)=t(x)$}. Let $\calM_{+}\models T_{+}$, and denote its reduct to $\calL$ by $\calM$. For $A\subseteq M$ we denote by 
$\langle A \rangle_{\Sk}$ the $\calL_{\Sk}$-substructure generated by $A$.\newline 

\noindent Following \cite{W75}, we say an $\calL_{\Sk}$-formula $\chi(x_1,\dots,x_n)$ is a \textbf{uniform configuration} if it is a conjunction of equalities of the form $f_{t}(x_{i_1}, \ldots ,x_{i_m})=x_{i_0}$ involving Skolem functions. We need the following result about uniform configurations from \cite{W75}.
\begin{fact}[{\cite[p. 448]{W75}}]\label{fact:ufconfigurations}
Let $\chi(x)$ be a uniform configuration. Then there exists an $\calL$-formula $\chi'(x)$ such that for all $\calA\models T_{+}$ and $a \in A^{|x|}$ the following are equivalent:
\begin{itemize}
\item $\calA \models \chi'(a)$.
\item The result of altering the Skolem structure of $\calA$ precisely so that $\calA \models \chi(a)$ is again a model of $T_+$. 
\end{itemize}
\end{fact}

\noindent In the case of Fact \ref{fact:ufconfigurations}, we say that $\chi'(x)$ \textbf{codes the eligibility of the configuration $\chi(x)$}.

\begin{lemma} \label{lem:setup}
Let $t_1(x),\dots,t_n(x)$ be $\calL_{\Sk}$-terms such that for every $i\leq n$ there is an $\mathcal{L}_{\Sk}$-function symbol $f_i$ with
\begin{equation*}
t_i(x) = f_i(x,t_1(x),\dots,t_{i-1}(x)).
\end{equation*}
Then there is an $\calL$-formula $\varphi(x,y)$ and a uniform configuration $\chi(x,y)$ such that
\[
T_+ \models \forall x \forall y \Big(\big(\varphi(x,y) \wedge \chi(x,y)\big) \leftrightarrow \big(\bigwedge_{i=1}^n y_i = f_i(x,y_1,\dots,y_{i-1})\big)\Big).
\]
\end{lemma} 
\begin{proof}
Let $J\subseteq \{1,\dots,n\}$ be the set of all $i$ such that $f_i \in \calL_{\Sk}\setminus \calL$. Let $\chi(x,y)$, where $y=(y_1,\dots,y_n)$, be the uniform configuration given by
\[
\bigwedge_{i\in J} f_i(x,y_1,\dots,y_{i-1}) = y_i
\]
and let $\varphi(x,y)$ be the $\calL$-formula given by
\[
\bigwedge_{i\in \{1,\dots,n\}\setminus J} f_i(x,y_1,\dots,y_{i-1}) = y_i.
\]
It is easy check this pair of formulas has the desired property.
\end{proof}

\noindent One of the main results in \cite{W75} is that if $T\models \operatorname{UF}$, then the Skolem expansion has a model companion. Indeed, more is true:

\begin{fact}[{\cite[Theorem 2, Corollary 3]{W75}}]\label{thm:winkler} Let $T\models \operatorname{UF}$. Then the Skolem expansion $T_{+}$ has a model companion $T_{\operatorname{Sk}}$ that satisfies $\operatorname{UF}.$
\end{fact}

\noindent \emph{From here on we assume that $T\models \operatorname{UF}$}. We have the following axiomatization of the model companion of the Skolem expansion.

\begin{fact}[{\cite[p. 447]{W75}}]\label{fact:axiom}
The theory $T_{\Sk}$ is axiomatized as the expansion of $T_+$ by the set $\Phi$ of all sentences of the form $\forall x_1 \ldots \forall x_k \psi(x)$, where $x=(x_1,\dots,x_n)$ and
\begin{enumerate}[(i)]
    \item $\psi(x)= \exists^{\infty} x_{k+1} \ldots x_n \varphi(x) \land \chi'(x) \rightarrow \exists x_{k+1}, \ldots x_n \varphi(x) \land \chi(x)$,
    \item $\varphi(x)$ is a quantifier free $\calL$-formula,
    \item $\chi(x)$ is a uniform configuration,
    \item $\chi'(x)$  codes the eligibility of the configuration $\chi(x)$.
\end{enumerate}
\end{fact}

\noindent Let $\calM_{\Sk}$ be an $|\mathcal{L}|^+$ -saturated model of $T_{\Sk}$ with underlying set $M$, and denote its reduct to $\calL$ by $\calM$. We need following easy corollary of the axiomatization of $T_{\Sk}.$

\begin{fact}\label{fact:bnf} Let $\mathcal{I}$ be the set of partial $\calL$-elementary maps $\iota: X \to Y$ between $\calM_{\Sk}$ and itself such that 
\begin{itemize}
    \item $\iota$ is a partial $\calL_{\Sk}$-isomorphisms and
    \item $X=\langle X\rangle_{\Sk}$ and $Y =\langle Y \rangle_{\Sk}$.
\end{itemize}
Then $\mathcal{I}$ is a back-and-forth system.
\end{fact}
\begin{proof}
Let $\iota : X\to Y$ in $\mathcal I$. Let $a \in M\setminus X$. By symmetry, it is enough to find $a'\in M$ such that there exists $\iota'\in \mathcal{I}$ extending $\iota$ such that $\iota(a)=a'$. By saturation of $\mathcal{M}_{\Sk}$, we just need to find $a' \in M$ such that for all $\mathcal{L}_{\Sk}(X)$-terms $t(x)=(t_1(x),\dots, t_n(x))$
\[
  \tp_{\mathcal{L}}(a',t(a')|Y)=\iota \tp_{\mathcal{L}}(a,t(a)|X).
\]

\noindent Without loss of generality, we can assume that there is $c\in X^m$ such that for every $i\in \{1,\dots, n\}$ there is a function symbol $f_i \in \calL_{\Sk}$ with
\begin{equation*}
t_i(x) = f_i(x,t_1(x),\dots,t_{i-1}(x),c).
\end{equation*}
Let $\varphi(x,y_1,\dots,y_n,z)$ be the $\calL$-formula and $\chi(x,y_1,\dots,y_n,z)$ be the  uniform configuration given by Lemma \ref{lem:setup}. Let the $\calL$-formula $\chi'(x,y_1,\dots,y_n,z)$ code the eligibility of $\chi(x,y_1,\dots,y_n,z)$. For ease of notation, set $y := (y_1,\dots,y_n)$.\newline

\noindent Consider an $\calL$-formula $\psi(x,y,z)$ and $c'\in X^{|c'|}
$ such that $\psi(x,y,c') \in \tp_{\mathcal{L}}(a,t(a)|X)$.
Extending $c$, we can assume that $c=c'$. By saturation of $\calM_{\Sk}$ it suffices to find $a' \in M$ such that $\calM_{\Sk} \models \psi(a',t(a'),\iota(c))$.
Since $\mathcal{M} \models \psi(a,t(a),c) \wedge \varphi(a,t(a),c) \wedge \chi'(a,t(a),c)$ and $a\not \in X$, we have that 
\[
\mathcal{M} \models \exists^{\infty} xy \ \psi(x,y,c) \wedge \varphi(x,y,c) \wedge \chi'(x,y,c).
\]
Since $\iota$ is $\calL$-elementary,
\[
\mathcal{M} \models \exists^{\infty} x y \ \psi(x,y,\iota(c)) \wedge \varphi(x,y,\iota(c)) \wedge \chi'(x,y,\iota(c))).
\]
Thus from the axiomatization of $T_{\Sk}$ we know that there is $(a',a_1',\dots,a_n')\in M^{1+n}$ such that
\[
\mathcal{M} \models \psi(a',a_1',\dots,a_n',\iota(c)) \wedge  \varphi(a',a_1',\dots,a_n',\iota(c)) \wedge \chi(a',a_1',\dots,a_n',\iota(c)).
\]
By our choice of $\varphi$ and $\chi$, we have that $a_i'=t_i(a')$ for each $i$. Thus $\calM \models \psi(a',t(a'),c).$
\end{proof}

\noindent We now collect the following easy corollary of Fact \ref{fact:bnf}.

\begin{fact}\label{fact:sametype}
Let $a,a'\in M^n$ and let $\sigma: M \to M$ be an $\calL_{\Sk}$-automorphism fixing $C$ such that $\sigma(a)=a'$ and for all $\mathcal{L}_{\Sk}$-terms $t(x)=(t_1(x),\dots, t_n(x))$
\[
\tp_{\mathcal{L}}(t(a)|C) = \tp_{\mathcal{L}}(t(a')|C).
\]
Then $\operatorname{tp}_{\calL_{\Sk}}(a|C)=\operatorname{tp}_{\calL_{\Sk}}(a'|C)$.
\end{fact}

\subsection{No new definable open sets in $T_{\Sk}$}
Let $\calM_{\Sk}$ be an $|\calL_{\Sk}|^+$-saturated model of $T_{\Sk}$ with underlying set $M$, and denote its reduct to $\calL$ by $\calM$. Fix a subset $C\subseteq M$ of cardinality at most $|\calL_{\Sk}|$.

\begin{thm}\label{thm:ocinsat}
Let $C=\langle C \rangle_{\Sk}.$ Then every open set definable over $C$ in $\mathcal{M}_{\Sk}$ is definable in $\mathcal{M}$.
\end{thm}
\begin{proof}
By \cite[Therorem 2.2]{BH12} it is enough to show that for every $a\in M^n$ for which the set of realisations of $\operatorname{tp}_{\calL}(a|C)$ is dense in an open set, the set of realisations of $\operatorname{tp}_{\calL_{\Sk}}(a|C)$ is dense in the set of realizations of $\operatorname{tp}_{\calL}(a|C)$. Let $U\subseteq M^n$ be an open definable set such that the set of realizations of $\operatorname{tp}_{\calL}(a|C)$ intersected with $U$ is dense in $U$. It is left to show that there is $a'\in U$ such that $a'\models \tp_{\calL_{\Sk}}(a|C)$. By Fact \ref{fact:sametype} and saturation of $\calM_{\Sk}$, it is enough to find for 
\begin{itemize}
    \item every tuple $t=(t_1,\dots, t_m): M^n \to M^m$ of $\calL_{\Sk}(C)$-terms and
    \item every $\mathcal{L}(C)$-definable set $X\subseteq M^{n+m}$ with $(a,t(a)) \in X$ 
    \end{itemize}
an $a'\in U$ such that $(a',t(a'))\in X$. Fix $t$ and $X$. After increasing $m$, we can assume that there is $c\in C^{\ell}$ such that $X$ is $\mathcal{L}(c)$-definable and for every $i\leq m$
\begin{equation*}
t_i(x) = f_i(x,t_1(x),\dots,t_{i-1}(x),c)    
\end{equation*}
where $f_i$ is a function symbol in $\calL_{\Sk}$. Let $\varphi(x,y_1,\dots,y_n,c)$ be the $\calL$-formula and $\chi(x,y_1,\dots,y_n,c)$ be the  uniform configuration given by Lemma \ref{lem:setup}. Set $y=(y_1,\dots,y_n)$.
Let the $\calL$-formula $\chi'(x,y,c)$ code the eligibility of $\chi(x,y,c)$.\newline

\noindent We now prove the existence of $a'$. Let $d_0$ be a realization of $\operatorname{tp}_{\calL}(a|C)$ in $U$. 
Let $d_1,\dots,d_m\in M^m$ be such that $(d_0,d_1,\dots,d_m) \in X$ and 
\[
\mathcal M\models (\varphi \wedge \chi')(d_0,d_1,\dots,d_m,c).
\]
Since there are infinitely many realizations of $\tp_{\mathcal{L}}(a|C)$ in $U$, there are infinitely many $e\in M^{n+m}$ such that $e\in X\cap U$ and $\mathcal{M}\models (\varphi\wedge \chi')(e,c)$.
Thus by Fact \ref{fact:axiom}, there is $e=(e_0,e_1,\dots,e_m)\in M^{n+m}$ such that 
\[
(e_0,e_1\dots ,e_m)\in X \cap U \text{ and } \mathcal{M}_{\Sk} \models \chi(e_0,e_1\dots ,e_m,c).
\]
Thus $(e_1,\dots,e_m)=t(e_0)$ and we can set $a'=e_0$.
\end{proof}

\begin{cor}\label{cor:opencore}
 Let $T'$ be an open core of $T$. 
Then $T'$ is an open core of $T_{\operatorname{Sk}}$.
\end{cor}
\begin{proof}
Let $\mathcal{L}'$ be the language of $T'$. Without loss of generality, we can assume that $\mathcal{L}'\cap \mathcal{L}_{\Sk}=\emptyset$.
Let $\mathcal{L}^*$ be the union of $\mathcal{L'} $ and $\mathcal{L}_{\Sk}$. Let $\mathcal{M}_{\Sk}\models T_{\Sk}$. Since $T'$ is an open core of $T$, we can expand $\mathcal{M}_{\Sk}$ to a model $\mathcal{M}^*$ of the $\mathcal{L}^*$-theory $T'\cup T_{\Sk}$. Let $X\subseteq M^n$ be an open set given by
\[
X:=\{ a \in M^n \ : \ \mathcal{M}_{\Sk} \models \varphi(a,c) \},
\]
where $\varphi$ is an
$\mathcal{L}_{\Sk}$-formula with parameters $c\in M^m$. Let $\mathcal{N}$ be an elementary extension of $\mathcal{M}^*$ that is $|\mathcal{L}|^+$-saturated. Set 
$
Y := \{ a\in N^n \ : \ \mathcal{N} \models \varphi(a,c)\}.$
Since $X$ is open, so is $Y$.
By Theorem \ref{thm:ocinsat} there is an $\mathcal{L}'$-formula $\psi(x,y)$ such that there is $d\in M^{\ell}$ with
$
Y = \{ a \in N^n \ : \ \mathcal{N} \models \psi(a,d)\}.
$
Since $\mathcal{M}^* \preceq \mathcal{N}$, there is $d'\in M^{\ell}$ such that 
$X = \{ a \in M^n \ : \ \mathcal{M}^* \models \psi(a,d')\}$.
Thus $X$ is $\mathcal{L}'$-definable. 
\end{proof}

\begin{proof}[Proof of Theorem \ref{thm:sklmf}] We are now able to complete the proof of Theorem \ref{thm:sklmf} using the same argument as in \cite[Corollary 4.9]{KR18}. Suppose $T\models \operatorname{DC}+\operatorname{UF}$ and let $T'$ be an open core of $T.$ Set $T_0$ be the Morleyization of $T$ in a language $\mathcal{L}_0$. For every $n> 0$, we will now construct a language $\mathcal{L}_n$ and an $\mathcal{L}_n$-theory $T_n$ such that
\begin{enumerate}
    \item $T_n$ has quantifier-elimination,
    \item $T_n\models \operatorname{UF}$, and
    \item $T'$ is an open core of $T_n$.
\end{enumerate}
\noindent Let $n\geq 0$, and suppose we already constructed a language $\mathcal{L}_n$ and an $\mathcal{L}_n$-theory $T_n$ with the properties (1)-(3). Let $\Phi$ be the set of $\mathcal{L}_n$-formulas
$\varphi(x,y)$ such that $|y|=1$ and
\[
T_n\models \forall x \exists ! y \ \varphi(x,y),
\]
For each $\varphi(x,y) \in \Phi$ we introduce a new function symbol $f_{\varphi}$ of arity $|x|$. Let $\tilde{\mathcal{L}}$ be the union of the $\mathcal{L}_{n}$ and $\{ f_{\varphi}\ : \ \varphi \in \Phi\}$. Let $\tilde{T}$ be the union of $T_{n}$ with the set of all $\tilde{\mathcal{L}}$-sentence of the form
\[
\forall x \forall y (f_{\varphi}(x)=y) \leftrightarrow \varphi(x,y),
\]
where $\varphi \in \Phi$. Since $\tilde{T}$ is an expansion of $T_n$ by definitions, it is easy to check that $\tilde{T}$ satisfies (1)-(3). Now consider the model companion $(\tilde{T})_{\Sk}$ of the Skolem expansion $(\tilde{T})_+$. Let $T_{n+1}$ be the Morleyization of $(\tilde{T})_{\Sk}$ in an expanded language $\mathcal{L}_{n+1}$. We know $T_{n+1}\models \operatorname{UF}$ by Fact \ref{thm:winkler}. By Corollary \ref{cor:opencore}, the theory $T'$ is an open core of $T_{n+1}$.\newline

\noindent Now set $T_{\Sk}^{\infty}:=\bigcup_{i\in \N} T_n$. From the construction, it follows immediately that $T'$ is an open core of $T_{\Sk}^{\infty}$ and that $T_{\Sk}^{\infty}$ has definable Skolem functions.
\end{proof}


\begin{thebibliography}{10}

\bibitem{BH12}
Gareth Boxall and Philipp Hieronymi,
\newblock Expansions which introduce no new open sets.
\newblock {\em J. Symb. Log.}, 77(1):111-121, (2012).


\bibitem{DMS}
Alfred Dolich, Chris Miller, and Charles Steinhorn,
\newblock Structures having o-minimal open core.
\newblock {\em Trans. Am. Math. Soc.}, 362(3):1371--1411, (2010).

\bibitem{DMS2}
Alfred Dolich, Chris Miller, and Charles Steinhorn,
\newblock Expansions of o-minimal structures by dense independent sets.
\newblock {\em Ann. Pure Appl. Logic}, 167(8):684--706, (2016).




\bibitem{KR18}
Alex Kruckman and Nicholas Ramsey,
\newblock Generic expansion and Skolemization in NSOP1 theories
\newblock {\em Ann. Pure Appl. Logic.}, 169(8): 755-774, (2018).



\bibitem{MS99}
Chris Miller and Patrick Speissegger,
\newblock  Expansions of the real line by open sets: o-minimality and open cores.
\newblock {\em Fund. Math.},162(3):193--208 (1999).



\bibitem{Shelah}
Saharon Shelah.
\newblock Stability, the f.c.p., and superstability; model theoretic properties of formulas in first order theory,
\newblock {\em Ann. Pure Appl. Logic}, 3 (3):  271-362 (1971)




\bibitem{Simon-DP}
Pierre Simon,
\newblock On dp-minimal structures
\newblock {\em J. Symb. Logic}, 76(2): 448--460 (2011)

\bibitem{S15}
Pierre Simon,
\newblock A guide to {NIP} theories.
\newblock {\em Lect. Notes Log.},
  vol.~44, Cambridge University Press, 2015.

\bibitem{W75}
Peter Winkler,
\newblock Model completeness and skolem expansions.
\newblock in {\em Model Theory and Algebra}, Springer (pp. 408--463), 1975.

\end{thebibliography}
\end{document}